\documentclass[12pt]{amsart}
\usepackage{amssymb}

\newtheorem{theorem}{Theorem}[section]
\newtheorem{proposition}[theorem]{Proposition}

\numberwithin{equation}{section}

\begin{document}

\title[Vector bundles on symmetric power]{Reconstructing vector 
bundles on curves from their direct image on symmetric powers}

\author[I. Biswas]{Indranil Biswas}

\address{School of Mathematics, Tata Institute of Fundamental
Research, Homi Bhabha Road, Mumbai 400005, India}

\email{indranil@math.tifr.res.in}

\author[D. S. Nagaraj]{D. S. Nagaraj}

\address{The Institute of Mathematical Sciences, CIT
Campus, Taramani, Chennai 600113, India}

\email{dsn@imsc.res.in}



\date{}

\begin{abstract}
Let $C$ be an irreducible smooth complex projective curve, and let $E$ be an 
algebraic vector bundle of rank $r$ on $C$. Associated to $E$, there are 
vector bundles ${\mathcal F}_n(E)$ of rank $nr$ on $S^n(C)$, where $S^n(C)$ is the 
$n$-th symmetric power of $C$. We prove the following: Let $E_1$ and $E_2$ be 
two semistable vector bundles on $C$, with ${\rm genus}(C)\, \geq\, 2$. If 
${\mathcal F}_n(E_1)\,\simeq \, {\mathcal F}_n(E_2)$ for a fixed
$n$, then $E_1 \,\simeq\, E_2$.
\end{abstract}

\maketitle

\section{Introduction} 

Let $C$ be an irreducible smooth projective curve defined over the field of
complex numbers. Let $E$ be a vector bundle of rank $r$ on $C$. Let
$S^n(C)$ be $n$-th symmetric power of $C$. Let $q_1$ (respectively,
$q_2$) be the projection of $S^n(C)\times C$ to $S^n (C)$
(respectively, $C$). Let $\Delta_n \subset S^n(C)\times C$
be the universal effective divisor of degree $n.$ The direct image
$${\mathcal F}_n(E ) \,:=\, q_{1\star} (q_{2}^{\star}(E)\vert_{\Delta_n})$$
is a vector bundle of rank $nr$ over $S^n(C)$. These vector bundles 
${\mathcal F}_n(E)$ are extensively studied (see \cite{ACGH}, \cite{BL},
\cite{BP}, \cite{ELN}, \cite{Sc}).

Assume that ${\rm genus}(C)\, \geq\, 2$. We prove the following (see Theorem
\ref{oneone2}):

\begin{theorem}
Let $E_1$ and $E_2$ be semistable vector bundles on $C$. If the two vector bundles
${\mathcal F}_n(E_1)$ and ${\mathcal F}_n(E_2)$ on $S^n(C)$ are isomorphic for a fixed
$n$, then $E_1$ is isomorphic to $E_2$.
\end{theorem}

\section{preliminaries}

Fix an integer $n\, \geq\, 2$. Let $S_n$ be the group of permutations of
$\{1\, ,\cdots\, ,n\}$. Given an irreducible smooth complex projective curve 
$C$, the group $S_n$ acts on $C^n$, and the quotient $S^n(C)\, :=\, C^n/{S_n}$ 
is an irreducible smooth complex projective variety of dimension $n$.

An effective divisor of degree $n$ on $C$ is a formal sum of the form 
$\sum_{i=1}^rn_i z_i$, where $z_i$ are points on $C$ and $n_i$ are positive 
integers, such that $\sum_{i=1}^rn_i\,=\,n$. The set of all effective divisors 
of degree $n$ on $C$ is naturally identified with $S^n(C)$.

Let $q_1$ (respectively, $q_2$) be the projection of $S^n(C)\times C$ on to 
$S^n(C)$ (respectively, $C$). Define
$$\Delta_n\, :=\, \{ (D,z) \,\in\, S^n(C)\times C~\mid ~ z \in D \}\,
\subset\, S^n(C)\times C\, .$$
Then $\Delta_n$ is a smooth hypersurface on $S^n(C)\times C$; it is
called the \textit{universal effective divisor} of degree $n$ of $C$.
The restriction of the projection $q_1$ to $\Delta_n$ is
a finite morphism
\begin{equation}\label{q}
q\, :=\, q_1\vert_{\Delta_n} \,:\, \Delta_n \,\longrightarrow\, S^n(C)
\end{equation}
of degree $n$. 

Let $E$ be a vector bundle on $C$ of rank $r$. Define
$$
{\mathcal F}_n(E ) \,:=\, 
q_{{\star}} (q_{2}^{\star}(E)\vert_{\Delta_n})
$$
to be the vector bundle on $S^n(C)$ of rank $nr$.

The \textit{slope} $E$ is defined to be $ \mu(E) \,:=\, \text{degree}(E)/r$.
The vector bundle $E$ is said to be {\em semistable} if $\mu(F)\,\leq\, \mu(E)$
for every nonzero subbundle $F$ of $E$.

\section{The reconstruction}

Henceforth, we assume that ${\rm genus}(C)\,\geq\, 2$.

We first consider the case of $n\,=\,2$. The hypersurface $\Delta_2$ in $S^2(C) 
\times C$ can be identified with $C\times C$. In fact the map $(x,y)\, 
\longmapsto\, (x+y,x)$ is an isomorphism from $C\times C$ to $\Delta_2$ (cf. 
\cite{BP}). Let
$$
q \,:\, \Delta_2 \, \longrightarrow \,S^2(C)
$$
be the map in \eqref{q}. Under the above identification of $\Delta_2$ with 
$C\times C$, the map $q$ coincides with the quotient map
\begin{equation}\label{f}
f\, :\, C \times C\, \longrightarrow\, S^2(C)\, =\, (C \times C)/S_2\, .
\end{equation}
For $i\, =\, 1\, ,2$, let 
$$p_i\,:\,C \times C \, \longrightarrow \,C$$ be the projection to the $i$-th 
factor. The diagonal $\Delta\, \subset\, C \times C$ is canonically isomorphic 
to $C$, and hence any vector bundle on $C$ can also be thought of as a vector 
bundle on $\Delta$.
For a vector bundle $E$ on $C$, we have the short exact sequence
$$
0\, \longrightarrow \, V(E)\,:=\, f^\star{\mathcal F}_2(E)\, \longrightarrow \, 
p_1^\star E \oplus p_2^\star E
\stackrel{q}{\longrightarrow} E \, \longrightarrow \, 0\, ,
$$
where $f$ is defined in \eqref{f}, and $q$ is the homomorphism defined by $(u,v) 
\,\longmapsto\, u-v$ (cf. \cite{BP}). Let $$\phi_i \,:\, V(E)\,\longrightarrow 
\,p_i(E)\, , ~\, i\,=\, 1\, ,2\, ,$$ be the restriction of the projection 
$p_1^\star E \oplus 
p_2^\star E\,\longrightarrow\, p_i^\star E$ to $V(E)\, \subset\, 
p_1^\star E\oplus p_2^\star E$. 
We have the following two exact sequences:
\begin{equation}\label{eq2}
0 \, \longrightarrow \, (p_2^\star E)(-\Delta)\,:=\,
(p_2^\star E)\otimes {\mathcal O}_{C\times C}(-\Delta)\, \longrightarrow \, V(E) 
\,\stackrel{\phi_1}{\longrightarrow}\, p_1^\star E \, \longrightarrow \, 0
\end{equation}
(we are using the fact that the restriction of the line bundle ${\mathcal 
O}_{C\times C}(-\Delta)$ to $\Delta$ is $K_\Delta \,=\, K_C$, where
$K_\Delta$ and $K_C$ are the canonical line bundles of $\Delta$ and $C$
respectively) and
$$
0 \, \longrightarrow \, (p_1^\star E)(-\Delta)\,:=\, 
(p_1^\star E)\otimes {\mathcal O}_{C\times C}(-\Delta)\, \longrightarrow \,V(E) 
\,\stackrel{\phi_2}{\longrightarrow} p_2^\star E\, \longrightarrow \, 0\, .
$$
 
\begin{proposition}\label{oneone}
Let $E$ and $F$ be two semistable vector bundles on $C$ such that 
${\mathcal F}_2(E) \,\simeq \,{\mathcal F}_2(F)$. Then $E$ is isomorphic to $F$.
\end{proposition}

\begin{proof}
The restriction of the exact sequence in \eqref{eq2} to the diagonal 
$\Delta\,=\, C$ gives a short exact sequence of vector bundles on $C$:
\begin{equation}\label{j1}
0\, \longrightarrow \, E\otimes K_C\, \longrightarrow \, J^1(E) \, 
\longrightarrow \,E\, \longrightarrow \, 0\, ,
\end{equation}
where $K_C$ is the canonical bundle on $C$. Similarly we have a short exact 
sequence 
\begin{equation}\label{j2}
0\,\longrightarrow\, F \otimes K_C\, \longrightarrow\, J^1(F)\, \longrightarrow 
\, F\, \longrightarrow \, 0\, .
\end{equation}
Since ${\mathcal F}_2(E)\,\simeq\,{\mathcal F}_2(F)$, we see that $J^1(E)\simeq J^1(F)$.
As $E$ (respectively, $F$) is semistable, and $\text{degree}(K_C) \, >\, 0$, the 
subbundle $E\otimes K_C$ (respectively, $F\otimes K_C$) of $J^1(E)$ 
(respectively, $J^1(F)$) in \eqref{j1} (respectively, \eqref{j2}) is the first 
term in the Harder--Narasimhan filtration of $J^1(E)$ (respectively, $J^1(F)$).
Since $J^1(E)\,\simeq\, J^1(F)$, this implies that that $E \,\simeq\, F$.
\end{proof}

Now we consider the general case of $n\, \geq\, 3$.

\begin{theorem}\label{oneone2}
Let $E$ and $F$ be semistable vector bundles on $C$ such that
${\mathcal F}_n(E) \,\simeq
\, {\mathcal F}_n(F)$. Then the vector bundle $E$ is isomorphic to $F$.
\end{theorem}
 
\begin{proof}
The universal effective divisor $\Delta_n\,\subset\, S^n(C) \times C$ of degree 
$n$ can be identified with $S^{n-1}(C) \times C$ using the morphism
$$
f\,:\, S^{n-1}(C) \times C \,\longrightarrow \, \Delta_n
$$
that sends any
$(D\, ,z) \, \in\, S^{n-1}(C) \times C$ to $(D+z,z)$. The composition
\begin{equation}\label{bq}
S^{n-1}(C) \times C \,\stackrel{f}{\longrightarrow} \, \Delta_n
\,\stackrel{q}{\longrightarrow} \, S^n(C)\, ,
\end{equation}
where $q$ is defined in \eqref{q}, will be denoted by $\overline{q}$. We note
that
$$
{\mathcal F}_n(E)\,=\, \overline{q}_\star p_2^\star E\, ,
$$
where $p_2\,:\, S^{n-1}(C)\times C\,\longrightarrow\, C$ is the natural 
projection. 

Let $f_1\,:\, S^{n-1}(C)\times C\,\longrightarrow\, S^{n-1}(C)$ is the natural
projection. Let
$$
\alpha\,:\, C\times C \,\longrightarrow\,  S^{n-1}(C) \times C
$$
be the 
morphism defined by $(x\, ,y)\,\longmapsto\, ((n-1)x\, , y)$. Then the pullback
$(\overline{q}\circ\alpha)^\star ({\mathcal F}_n(E))$, where $\overline{q}$ is 
constructed in \eqref{bq}, fits in an exact sequence:
\begin{equation}\label{ex1}
0\,\longrightarrow\,
p_2^\star E \otimes{\mathcal O}_{C\times C}(-(n-1)\Delta)\,\longrightarrow\,
(\overline{q}\circ\alpha)^\star ({\mathcal F}_n(E))\,\longrightarrow
(f_1\circ\alpha)^\star {\mathcal F}_{n-1}(E)\,\longrightarrow\, 0\, ,
\end{equation}
where $f_1$ is defined above. The above projection
$(\overline{q}\circ\alpha)^\star ({\mathcal F}_n(E))\,\longrightarrow
(f_1\circ\alpha)^\star{\mathcal F}_{n-1}(E)$ follows from the fact that
$f_1\circ\alpha (z) \, \subset\, \overline{q}\circ\alpha (z)$ 
for any $z\, \in\, C\times C$.

Define the vector bundle
$$J^{n-1}(E)\,:=\,(f\circ\alpha)^\star
({\mathcal F}_n(E))\vert_{\Delta}\,\longrightarrow\, 
\Delta\, =\, C$$
on the diagonal in $C\times C$. Restricting the exact sequence in \eqref{ex1} to 
$\Delta$, we get a short exact sequence of vector bundles
$$
0\, \longrightarrow\, J^{n-2}(E)\otimes K_C \, \longrightarrow\,
J^{n-1}(E)\, \longrightarrow\, E\, \longrightarrow\, 0\, ;
$$
note that $J^{0}(E)\, =\, E$.

Therefore, by induction on $n$, we get a filtration of subbundles of
$J^{n-1}(E)$
\begin{equation}\label{fi}
0\,=\, W_n\, \subset\, W_{n-1}\, \subset\, \cdots \, \subset\,
W_1\, \subset\, W_0\, =\, J^{n-1}(E)\, ,
\end{equation}
such that $W_j/W_{j+1} \,=\, E\otimes K_C^{\otimes j}$ for all $j\, \in\, [0\, ,
n-1]$. In particular, $W_{n-1}\,=\, E\otimes K_C^{\otimes (n-1)}$.

Since $E$ is semistable, and $\text{degree}(K_C)\, >\, 0$, we 
conclude that $E\otimes K_C^{\otimes j}$ is semistable for all
$j\, \in\, [0\, , n-1]$, and
$$
\mu(W_j/W_{j+1})\, <\, \mu(W_{j+1}/W_{j+2})
$$
for all $j\, \in\, [0\, , n-2]$. Consequently, the filtration of $J^{n-1}(E)$ 
in \eqref{fi} coincides with the Harder--Narasimhan filtration of $J^{n-1}(E)$. 
In particular, the first term of the Harder--Narasimhan filtration (the maximal 
semistable subsheaf) of $J^{n-1}(E)$ is the subbundle $E\otimes K_C^{\otimes 
(n-1)}$.

Using this, and the fact that $J^{n-1}(E)\,\simeq\,J^{n-1}(F)$ (recall that 
${\mathcal F}_n(E)\,\simeq\,{\mathcal F}_n(F)$), we conclude that
$F$ is isomorphic to $E\,=\, 
W_0/W_1$.
\end{proof}

\end{document}